\documentclass{amsproc}

\newtheorem{theorem}{Theorem}[section]
\newtheorem{lemma}[theorem]{Lemma}
\newtheorem{corollary}[theorem]{Corollary}
\theoremstyle{definition}

\theoremstyle{remark}

\numberwithin{equation}{section}

\begin{document}

\title{Increasing stability for near field from the scattering amplitude}

\author{Victor Isakov}
\address{ Department of Mathematics, Statistics, and Physics, 
Wichita State University, Wichita, KS 67260-0033}

\email{victor.isakov@wichita.edu}
\thanks{This research was in part supported by the NSF grant DMS 10-08902 and by Emylou Keith and Betty 
Dutcher Distinguished Professorship at WSU}

%    General info
\date{December 20, 2013.}


\begin{abstract}
We obtain stability estimates for the near field of a radiating solution of the Helmholtz equation from the far field (scattering amplitude). This estimates contain a best possible Lipschitz term, a H\"older term, and terms which decay as powers
of the frequency $k$ for large $k$ under some a priori bounds. These estimates contain only explicit constants and show increasing stability of recovery of the  near field from scattering amplitude with growing $k$. Proofs are elementary and are based on new explicit bounds for Hankel functions. We give first applications to increasing stability in (linearized) inverse scattering by obstacles.  
\end{abstract}

\maketitle

\section{Introduction.}

Many inverse problems are known to be severely ill-posed. This  makes it extremely difficult to design reliable reconstruction algorithms  and dramatically restricts their resolution in practice. However, in some cases, it has been observed numerically that the stability increases with respect to some parameter such as the wave number (or energy). Ill-posedness occurs at the stage of the continuation of solutions of partial differential equations from observation set toward
an obstacle. Several rigorous justifications of the increasing stability phenomena in the Cauchy (or continuation) problem in different settings were obtained by Isakov {\it et al} \cite{HI, I07, IK, ASI10}. These justifications are in form of conditional stability estimates which are getting nearly Lipschitz when the wave number $k$ is getting large. 

 For increasing stability for the Schr\"odinger potential from the Dirichlet-to-Neumann map we refer to \cite{I11},  \cite{IN}, \cite{INUW}, and \cite{IW}. As a an important example of (at least generically and locally) well-posed inverse scattering problem we mention the inverse backscattering problem \cite{ER}.

We consider a solution $u$ to the Helmholtz equation
\begin{equation}
\label{H}
(\Delta + k^2)u=0\;\mbox{in}\;{\mathbb R}^3\setminus \bar D,
\end{equation}
satisfying  the Sommerfeld radiation condition
\begin{equation}
\label{Som}
 lim r(\partial_r u -iku)(x)  = 0 \;\mbox{as}\;r\rightarrow\infty.
  \end{equation}
  Here $r=|x|$. As well known \cite{CK}, \cite{LP},
  \cite{T} the relations \eqref{H}, \eqref{Som} imply that
  $$
  u(x)= \frac{e^{ikr}}{r} A(\sigma)+ O(r^{-2}),
  $$
  where $\sigma= r^{-1} x$ and $A(\sigma)$ is the so called scattering amplitude (or pattern). We are interested in recovery of $u$ from $A$. It is known that $A$ is a (real) analytic function on the unit sphere. Uniqueness of $u$ follows from well know Rellich Theorem. Our main goal is to study stability of this recovery.
  
  In this paper we will use denote by $B_R$ the ball $\{x: |x|<R\}$ in ${\mathbb R}^3$.
  
To state our results we use complete orthonormal base in $L^2(B_1)$
formed of spherical harmonics $Y_n^m(\sigma), n=0,1,2,...,
m=-n,...,0,...,n$. Let $A\in L^2(B_R)$ and $a_{m,n}$ be the coefficients of the expansion of $A$ with respect to $Y_n^m$, i.e., $A = \sum_{n,m} a_{m,n} Y_n^m$. For brevity, we introduce
$$
Y_n(\cdot ;A) = 
a_n^{-1}
\sum_{m=-n,...,n}a_{m,n}Y_n^m, 
a_n=(\sum_{m=-n,...,n} |a_{m,n}|^2)^{\frac{1}{2}}
$$  
then
\begin{equation}
\label{exp}
A(\sigma)= \sum_{n=0}^{\infty} a_n Y_n(\sigma;A).
\end{equation}
Observe that $\|A\|_{(0)}^2=\sum_{n=0}^{\infty} |a_n|^2$.
As known \cite{CK}, \cite{T},
\begin{equation}
\label{expu}
u(x)= \sum_{n=0}^{\infty} u_n(r) Y_n(\sigma;A),\;
\mbox{where}\;
u_n(r)  = ki a_n h_n^{(1)}(kr).
\end{equation}

For a function $u$ with the expansion \eqref{expu} we will use the following (natural) Sobolev norm
\begin{equation}
\label{norm}
\| u\|^2_{(l)}(\partial B_R)= R^2 sum_{m=0}^l \sum_{n=0}^{\infty} (\frac{n}{R})^{2m} |u_n(R)|^2.
\end{equation}

We let $\varepsilon_1^2=\sum_{n=0}^N |a_n|^2, \varepsilon_2^2=\sum_{n=N+1}^{\infty} |a_n|^2$.

Now we state our main results where we
chosen $N=[\sqrt{kR}]$  and  
$E=-log \varepsilon_2$. Here $[a]$ is the entire part of $a$.

\begin{theorem}
\label{Th1}
Assume that $2\leq kR$.
 
Then we have the following stability estimates 
\begin{equation}
\label{T1}
\lVert u \rVert^2_{(0)}(\partial B_R)
\leq 
\frac{2 e^2}{\pi} \varepsilon_1^2 
+ \frac{2}{\pi}e^{\frac{2}{R}} \varepsilon_2+ 
 R^2\frac{M_1^2}{E+k}, 
\end{equation}
and
\begin{equation}
\label{T2}
\lVert u \rVert^2_{(0)}(\partial B_R)
\leq 
\frac{2 e^2}{\pi}  \varepsilon_1^2 
+ \sqrt{\frac{2R}{\pi k}}e^{\frac{1}{R}} M_1 \varepsilon_2^{\frac{1}{2}}+ 
R^2\frac{M_1^2}{E+k} , 
\end{equation}
where $M_1=\|u\|_{(1)}(\partial B_R)$. 
\end{theorem}

\begin{theorem}
\label{Th2}
Assume that $2\leq kR$.
 
Then we have the following stability estimate: 
$$
\lVert \partial_r u \rVert^2_{(0)}(\partial B_R)
\leq 
$$
\begin{equation}
\label{T1der}
\frac{ e^2}{\pi}(3+\sqrt{5}) k^2 \varepsilon_1^2 
+  k^2 e^{\frac{2}{R}}\varepsilon_2+ 
R^2\frac{M_2^2}{E+k-2\sqrt{E+k}+1} , 
\end{equation}
where  $M_2=\|\partial_r u\|_{(1)}(\partial B_R)$. 
\end{theorem}

From estimates \eqref{T1}, \eqref{T2} it is obvious that the stability behaves more like Lipschitz type when $k$ is large.
Indeed, the second and third terms on the right side \eqref{T2}  go to zero as powers of $k$, which quantifies the increasing stability. We observe that the bounds \eqref{T1}, \eqref{T2}, \eqref{T1der} are so called conditional stability estimates: they guarantee stability under a priori constrains of higher norms of solutions. Due to
ill-posedness of recovery of $u$ from $A$ (\cite{B},cite{CK},\cite{I}, \cite{T}) stability estimates are impossible without such constraints.

Known stability estimates for $u$ from its scattering amplitude \cite{B}, \cite{I}, \cite{T} are of logarithmic type, contain unknown constants, and do not indicate increasing stability for larger $k$.    

Our proofs are using well known expression \eqref{expu} of $u$ via the expansion \eqref{exp} of the scattering amplitude as the series in spherical harmonics. The crucial step is explicit upper bounds for the Hankel functions $h_n^{(1)}(t)$ given by Lemmas 2.1, 2.3 with surprisingly short and elementary proofs. Theory of Bessel and Hankel functions abound with basic, but hard open questions (about sharp maxima, zeros etc) \cite{W}. While some bounds (similar to  Lemma 2.2) and asymptotic behaviour of these functions are well known, explicit bounds when $0<t<n$ are only partially available and constants in these bounds are not explicit \cite{BRV}, Lemma 1, p. 364. Some refined properties of Bessel functions were used by
F. John \cite{J} to find a crucial example showing growing instability for the continuation of solution to the Helmholtz equation from the unit disk onto its complement in the plane. In \cite{IK}
by using energy integrals for the Bessel's equation
we demonstrated increasing stability of the continuation for the John's example in low frequency zone (which grows with $k$).

The paper is organized as follows. In Section~2 we will obtain some explicit bounds on Hankel functions $h_n^{(1)}$. 
 In Section~3 we present proofs of Theorem~\ref{Th1} and
 of  Theorem~\ref{Th2}. In Section~4 we give applications
 to the increasing stability of inverse obstacle scattering problem linearized around a sphere. Finally we 
discuss challenging open problems and possible further developments.

\section{Some bounds of Hankel functions}

We will use that
\begin{equation}
\label{h}
h_n^{(1)}(t) = \sqrt{\frac{2}{\pi}}i^n \frac{e^{i t}}{t}
\sum_{m=0}^n(-1)^n\frac{(n+m)!}{m! (n-m)!}\frac{i^m}{(2t)^m},
\end{equation}
provided $0<t$.
This is a well-known representation of the Hankel function
given for example in \cite{JEL}, p. 142, \cite{T}, p. 205, \cite{W}, p. 53.

To prove main results, we need  elementary but crucial lemmas.  

\begin{lemma}
\label{low}
If $n^2<t$, then 
\begin{equation}
\label{lowb}
|h_n^{(1)}(t)|<\frac{\sqrt{2}e}{\sqrt{\pi}t}.
\end{equation}
 \end{lemma}
\begin{proof}
Using \eqref{h} and the triangle inequality we yield
$$
|h_n^{(1)}(t)|   \leq 
 \frac{\sqrt{2}}{\sqrt{\pi} t}
\sum_{m=0}^n \frac{(n+m)!}{m! (n-m)!}\frac{1}{(2t)^m} =
$$
$$
\frac{\sqrt{2}}{\sqrt{\pi} t}
(1+\sum_{m=1}^n \frac{(n-m)!}{(n-m)!} \frac{1}{m!}
\frac{(n-m+1)...n(n+1)...(n+m)}{n^m (2n)^m}),
$$
where we used the assumption that $n^2<t$.  Since
$m \leq n$, we have $(n-m+1)...n\leq n^m$ and
$(n+1)...(n+m)\leq (2n)^m$, so continuing the bounds of
$|h_n^{(1)}|$ we obtain
$$
|h_n^{(1)}(t)| \leq  \frac{\sqrt{2}}{\sqrt{\pi} t}
\sum_{m=0}^n \frac{1}{m!} \leq 
 \frac{\sqrt{2} e}{\sqrt{\pi}t}.
$$
\end{proof}

\begin{lemma}
\label{global} 
If $0<t$, then
\begin{equation}
\label{globalb}
|h_n^{(1)}(t)| < \frac{\sqrt{2}}{\sqrt{\pi}t}
(1+\frac{n}{t})^n.
\end{equation}
 \end{lemma}
\begin{proof}
Again using \eqref{h} and the triangle inequality we yield
$$
|h_n^{(1)}(t)|   \leq 
 \frac{\sqrt{2}}{\sqrt{\pi} t}
\sum_{m=0}^n \frac{n!}{m! (n-m)!}(n+1)...(n+m)
\frac{1}{(2t)^m} \leq
$$
$$
\frac{\sqrt{2}}{\sqrt{\pi} t}
\sum_{m=0}^n \frac{n!}{m!(n-m)!} (\frac{n}{t})^m \leq
\frac{\sqrt{2}}{\sqrt{\pi} t} (1+ \frac{n}{t})^n ,
$$
due to the binomial formula.

\end{proof}

Now we similarly obtain bounds for derivatives of the Hankel functions.

\begin{lemma}
\label{lowder}
If $n^2<t$, then 
\begin{equation}
\label{lowb'}
|\partial_t h_n^{(1)}(t)|<\frac{\sqrt{2}e}{\sqrt{\pi}} \frac{\sqrt{t^2+1}+1}{t^2}.
\end{equation}
 \end{lemma}
\begin{proof}
Differentiating \eqref{h}  we yield
$$
\partial_t h_n^{(1)}(t)= 
\frac{\sqrt{2}}{\sqrt{\pi} } \frac{e^{it}}{t} 
 (\frac{it-1}{t}
\sum_{m=0}^n (-1)^n\frac{(n+m)!}{m! (n-m)!}
\frac{i^m}{(2t)^m} -
$$
\begin{equation}
\label{h'}
\sum_{m=1}^n (-1)^n\frac{(n+m)!}{m!(n-m)!} 
\frac{i^m}{(2t)^m}\frac{m}{t}).
\end{equation}

Since $|it-1|=\sqrt{t^2+1}$ as in the proof of Lemma 2.1 we obtain
\begin{equation}
\label{first}
|\frac{it-1}{t}\sum_{m=0}^n (-1)^n
\frac{(n+m)!}{m! (n-m)!}\frac{i^m}{(2t)^m}|
   \leq \frac{\sqrt{t^2+1}}{t} e.
\end{equation}
 For the second sum on the right side of \eqref{h'} we have
   $$
   |\sum_{m=1}^n (-1)^n\frac{(n+m)!}{m!(n-m)!} 
\frac{i^m}{(2t)^m}\frac{m}{t}| \leq
$$
\begin{equation}
\label{second}
\sum_{m=1}^n \frac{(n-m)!}{(n-m)!} \frac{1}{(m-1)!}
\frac{(n-m+1)...n(n+1)...(n+m)}{n^m (2n)^m}\frac{1}{t} \leq \frac{e}{t},
\end{equation}
where we used the assumption that $n^2<t$ and again followed the argument in Lemma 2.1. 

From \eqref{h'} by the triangle inequality with use of  \eqref{first}, and \eqref{second} we yield
$$
|\partial_t h_n^{(1)}(t)| \leq
\frac{\sqrt{2}}{\sqrt{\pi}}(\frac{\sqrt{t^2+1}}{t^2} e+ \frac{1}{t^2} e)
$$
and complete the proof of \eqref{lowder}.

\end{proof}

\begin{lemma}
\label{globalder} 
If $0<t$, then
\begin{equation}
\label{globalb'}
|h_n^{(1)}(t)| \leq \frac{\sqrt{2}}{\sqrt{\pi}t}
(\frac{\sqrt{t^2+1}}{t}+\frac{n}{t})
(1+\frac{n}{t})^n.
\end{equation}
 \end{lemma}
 
\begin{proof}

As in the proof of Lemma 2.3 we will bound two terms in \eqref{h'}. We have
\begin{equation}
\label{firstglobal}
|\frac{it-1}{t}\sum_{m=0}^n (-1)^n
\frac{(n+m)!}{m! (n-m)!}\frac{i^m}{(2t)^m}|
   \leq \frac{\sqrt{t^2+1}}{t} (1+\frac{n}{t})^n,
\end{equation}
by repeating the proof of Lemma 2.2.

 For the second sum on the right side of \eqref{h'} by the triangle inequality we have
   $$
   |\sum_{m=1}^n (-1)^n\frac{(n+m)!}{m!(n-m)!} 
\frac{i^m}{(2t)^m}\frac{m}{t}| \leq
$$
\begin{equation}
\label{secondglobal}
\frac{1}{t}\sum_{m=1}^n \frac{n!}{m!(n-m)!} 
(n+1)...(n+m)\frac{n}{(2t)^m} = 
\frac{n}{t} (1+\frac{n}{t})^n,
\end{equation}
where we again followed the argument in Lemma 2.1. 

Combining \eqref{h'}, \eqref{firstglobal}, and \eqref{secondglobal} we yield
$$
|\partial_t h_n^{(1)}(t)| \leq
\frac{\sqrt{2}}{\sqrt{\pi}t}
(\frac{\sqrt{t^2+1}}{t} (1+\frac{n}{t})^n+
\frac{n}{t}(1+\frac{n}{t})^{n+1}),
$$
and complete the proof of \eqref{globalb'}.

\end{proof}

\section{Proof of main results}

For a ($L^2$-) function $u$ on the sphere $\partial B_R$ we have the orthonormal expansion
$$
u(x)= \sum_{n=0}^{\infty} u_n Y_n(\sigma;u)
$$
and introduce the low frequency projector
$$
P_N u(x)= \sum_{n=0}^{N} u_n Y_n(\sigma;u).
$$

\begin{lemma} 
Assume that $2\leq kR$.
 
Then we have the following stability estimate: 
\begin{equation}
\label{lowbound}
\lVert P_N u \rVert_{(0)}(\partial B_R)
\leq \frac{\sqrt{2}}{\pi} e \varepsilon_1. 
\end{equation}
\end{lemma}

This result  follows from \eqref{lowb} and 
\eqref{norm}. It shows Lipschitz stability
of the low frequency part $P_N u$ from low frequency part of $A$. Since we choose
$N=[kR]$, this part well approximates $u$ when $k$ is large.  

Now we give a proof of Theorem 1.1

\begin{proof}
Using the representation \eqref{expu} we yield
$$
\|u\|^2_{(0)}(\partial B_R)= 
\int_{\partial B_R}|u|^2(x) d\Gamma(x)=
$$
$$
k^2\int_{\partial B_R}|\sum_{n=0}^{\infty} i^n  a_n 
h_n^{(1)}(kR)Y_n(\sigma)|^2 d\Gamma(x) =
$$
$$
k^2 R^2 \sum_{n=0}^{\infty} |a_n|^2 
|h_n^{(1)}(kR)|^2,
$$
due to orthonormality of the system $Y_n$ on the unit sphere.
We let $N_1=[\sqrt{E+k}]$ and consider two cases:
1) $N+1\leq N_1$ and 2) $N_1\leq N$.

In case 1) we split the last sum into three terms obtaining
$$
\|u\|^2_{(0)}(\partial B_R) =
k^2 R^2 ( \sum_{n=0}^{N} |a_n|^2 
|h_n^{(1)}(kR)|^2 + 
$$
\begin{equation}
\label{split}
\sum_{n=N+1}^{N_1} |a_n|^2 
|h_n^{(1)}(kR)|^2 +
\sum_{n=N_1+1}^{\infty} |a_n|^2 
|h_n^{(1)}(kR)|^2).
\end{equation}

We have
$$
(1+\frac{N_1}{kR})^{2N_1}\varepsilon_2^2=
e^{2N_1 log(1+\frac{N_1}{kR})-2E}\leq
$$
\begin{equation}
\label{N1}
e^{2\frac{N_1^2}{kR}-2E}\leq
e^{-E+\frac{2}{R}}= e^{\frac{2}{R}}\varepsilon_2,
\end{equation}
where we used that $log(1+x)<x$, when $0<x$.

Observe that
$$
k^2 R^2\sum_{n=N_1+1}^{\infty} |a_n|^2 |h_n^{(1)}(kR)|^2\leq
$$
\begin{equation}
\label{N1+1}
\frac{1}{(N_1+1)^2}k^2 R^2\sum_{n=N_1+1}^{\infty} n^2 |a_n|^2 
|h_n^{(1)}(kR)|^2 = R^2\frac{M_1^2}{(N_1+1)^2}
\end{equation}

Finally from \eqref{split} by using \eqref{lowb}, \eqref{globalb},\eqref{N1}, and \eqref{N1+1} we obtain
$$
\|u\|^2_{(0)}(\partial B_R) \leq
$$
$$
k^2R^2\frac{2e^2}{\pi (kR)^2} \varepsilon_1^2 +
k^2R^2 e^{\frac{2}{R}} \frac{2}{\pi k^2 R^2}\varepsilon_2 +
R^2\frac{M_1^2}{(N_1+1)^2}
$$
which gives \eqref{T1} because of the choice of $N_1$.

In case 2) instead of \eqref{split} we write
$$
\|u\|^2_{(0)}(\partial B_R) =
$$
\begin{equation}
\label{split1}
k^2 R^2 ( \sum_{n=0}^{N} |a_n|^2 
|h_n^{(1)}(kR)|^2 + 
\sum_{n=N+1}^{\infty} |a_n|^2 
|h_n^{(1)}(kR)|^2).
\end{equation}
As in \eqref{N1+1} we have
\begin{equation}
\label{N+1}
k^2\sum_{n=N+1}^{\infty} |a_n|^2 |h_n^{(1)}(kR)|^2\leq \frac{M_1^2}{(N+1)^2}
\end{equation}
Similarly to case 1), from \eqref{split1} by using \eqref{low} and \eqref{N+1} we yield
$$
\|u\|^2_{(0)}(\partial B_R) \leq
k^2R^2\frac{2e^2}{\pi (kR)^2} \varepsilon_1^2 + 
R^2\frac{M_1^2}{(N+1)^2} \leq
$$
$$
\frac{2e^2}{\pi} \varepsilon_1^2 + 
R^2\frac{M_1^2}{(N_1+1)^2} \leq \frac{2e^2}{\pi} \varepsilon_1^2 + 
R^2\frac{M_1^2}{E+k}
$$
because in case 2) $N_1\leq N$ and $E+k \leq(N_1+1)^2$.
So again \eqref{T1} follows. 

We consider the same two cases: 1) $N+1\leq N_1$ and
2) $N_1\leq N$.

In case 1), as in the previous proof, we have the equality \eqref{split}. Now we bound the 
second term on its right side in a different way:
$$
k^2 R^2\sum_{n=N+1}^{N_1} |a_n|^2 
|h_n^{(1)}(kR)|^2 =
\sum_{n=N+1}^{N_1} \frac{kR}{n}|a_n||h_n^{(1)}(kR)|
kRn|a_n||h_n^{(1)}(kR)| \leq
$$
\begin{equation}
\label{CS}
R(\sum_{n=N+1}^{N_1} \frac{k^2R^2}{n^2}|a_n|^2|h_n^{(1)}(kR)
|^2)^{\frac{1}{2}}
(\sum_{n=N+1}^{N_1} k^2n^2|a_n|^2|h_n^{(1)}|
(kR)^2)^{\frac{1}{2}}
\end{equation}
when we use the Cauchy-Schwarz inequality. Bounding the first term on the right side of \eqref{CS} via 
\eqref{globalb} and the second term from the definition of the Sobolev norm we yield
$$
k^2 R^2\sum_{n=N+1}^{N_1} |a_n|^2 |h_n^{(1)}(kR)|^2 \leq
$$
$$
\frac{\sqrt{2}}{\sqrt{\pi}(N+1)}((1+\frac{N_1}{kR})^{N_1} R (\sum_{n=N+1}^{N_1}
|a_n|^2)^{\frac{1}{2}} M_1 =
$$
\begin{equation}
\label{CS1}
\frac{\sqrt{2}R}{\sqrt{\pi}(N+1)}((1+\frac{N_1}{kR})^{N_1} R \varepsilon_2 M_1 \leq
\sqrt{\frac{2}{\pi}}R \frac{1}{\sqrt{kR}} e^{\frac{1}{R}} M_1 \varepsilon_2^{\frac{1}{2}},
\end{equation}
where we used \eqref{N1}.

From \eqref{split} by using  \eqref{CS1} we derive that
$$
\|u\|^2_{(0)}(\partial B_R) \leq
$$
$$
k^2R^2\frac{2e^2}{\pi (kR)^2} \varepsilon_1^2 +
\sqrt{\frac{2R}{\pi}} \frac{1}{\sqrt{k}} e^{\frac{1}{R}} M_1 \varepsilon_2^{\frac{1}{2}} +
R^2\frac{M_1^2}{(N_1+1)^2}
$$
which as above produces \eqref{T2}.

Case 2) is considered exactly as in Theorem 1.1. 

Now we choose $N_1=[E+k]-1$ and again consider two cases:1) $N+1\leq N_1$ and
2) $N_1\leq N$.

\end{proof}

Now we similarly prove Theorem 1.2.

\begin{proof}

Now we choose $N_1=[\sqrt{E+k}]-1$ and again consider two cases:1) $N+1\leq N_1$ and
2) $N_1\leq N$.

In case 1) similarly to \eqref{split}
$$
\|\partial_ru\|^2_{(0)}(\partial B_R) =
k^4 R^2 ( \sum_{n=0}^{N} |a_n|^2 
|\partial_th_n^{(1)}(kR)|^2 + 
$$
\begin{equation}
\label{splitder}
\sum_{n=N+1}^{N_1} |a_n|^2 
|\partial_t h_n^{(1)}(kR)|^2 +
\sum_{n=N_1+1}^{\infty} |a_n|^2 
|\partial_t h_n^{(1)}(kR)|^2).
\end{equation}
Using Lemma 2.4  and the obvious inequality $\frac{t^2+1}{t^2}\leq\frac{5}{4}$, provided $2\leq t$, we obtain
$$
|\partial_t h_n^{(1)}(kR)|^2 \leq
\frac{2}{\pi k^2 R^2} \frac{5}{4} (1+\frac{n}{kR})^{2(n+1)}\leq 
\frac{1}{ k^2 R^2} (1+\frac{n}{kR})^{2(n+1)}
$$
Similarly to \eqref{N1}
\begin{equation}
\label{N11}
(1+\frac{N_1}{kR})^{2(N_1+1)}\varepsilon_2^2 \leq
e^{2\frac{N_1(N_1+1)}{kR}-2E}\leq
e^{-E+\frac{2}{R}}= e^{\frac{2}{R}}\varepsilon_2,
\end{equation}
because $N_1=[\sqrt{E+k}]-1$.
As in \eqref{N1+1}
\begin{equation}
\label{N1+2}
k^4 R^2\sum_{n=N_1+1}^{\infty} |a_n|^2 |\partial_t h_n^{(1)}(kR)|^2\leq r^2\frac{M_2^2}{(N_1+1)^2}.
\end{equation}
Hence 
$$
\|\partial_ru\|^2_{(0)}(\partial B_R) \leq
\frac{2e^2}{\pi} \frac{\sqrt{5}+1}{2} k^2\varepsilon_1^2+
e^{\frac{2}{R}}k^2\varepsilon_2+
 R^2\frac{M_2^2}{(N_1+1)^2}
$$

Case 2) is considered exactly as in Theorem 1.1
by splitting into two terms instead if three
in \eqref{splitder}. 

\end{proof}

\section{Application to linearized inverse obstacle scattering}

We consider a solution $u_0$ of the simplest scattering problem
\begin{equation}
\label{uH}
\Delta u_0 + k^2 u_0=0 \;\mbox{in}\; 
{\mathbb R}^3 \setminus \bar D_0,
\end{equation}
with the Dirichlet boundary condition (soft obstacle)
\begin{equation}
\label{ub}
 u_0 = 1  \;\mbox{on}\; \partial D_0
\end{equation}
and with the Sommerfeld radiation condition \eqref{Som} for $u_0$. Here $D_0$ is a bounded
domain with $C^2$-boundary and with connected
complement of $\bar D$.

More important in applications is the hard obstacle problem where the Dirichlet boundary condition
\eqref{ub} is replaced with the Neumann condition
\begin{equation}
\label{ub1}
\partial_{\nu}u_1 = 1  \;\mbox{on}\; \partial D_0
\end{equation}
for the solution $u_1$ to the Helmholtz equation
\eqref{uH} with the radiation condition 
\eqref{Som}.

We will consider obstacle $D_0= B_R$ in 
${\mathbb R}^3$.
Then the scattering problem \eqref{uH}, \eqref{ub} has the explicit solution
\begin{equation}
\label{u0}
u_0(x)= \frac{R}{e^{ikR}} \frac{e^{ikr}}{r} .
\end{equation}
and the hard scattering problem has the solution
\begin{equation}
\label{u1}
u_1(x)= \frac{R^2}{(ikR-1)e^{ikR}}
 \frac{e^{ikr}}{r} .
\end{equation}

Observe that $u_0, u_1$ can be viewed as incident spherical waves.

Let $D =\{x: r<R+d(\sigma)\}$ where $d$ is a function on $\partial B_1$ with small norm in $C^2$.  It is known \cite{H} that the solution $u$ to the scattering problem \eqref{uH}, \eqref{ub} with $D_0$ replaced by $D$ is $u_0+v_0+...$ with the scattering amplitude $A_0+A(v_0)+...$ where $A_0$ is the scattering amplitude of $u_0$, $A(v_0)$ is the scattering amplitude of the solution $v_0$ to the following scattering problem
\begin{equation}
\label{vH}
\Delta v_0 + k^2 v_0=0 \;\mbox{in}\; 
{\mathbb R}^3 \setminus \bar D_0,
\end{equation}
\begin{equation}
\label{vb}
 v_0 = -d \partial_r u_0 \;\mbox{on}\; 
\partial D_0
\end{equation}
with the Sommerfeld condition \eqref{Som} for $v_0$. The term $...$ has the  norm bounded by $C\|d\|_0^2$.

The linearized hard obstacle problem is similarly
the following scattering problem
\begin{equation}
\label{v1H}
\Delta v_1 + k^2 v_1=0 \;\mbox{in}\; 
{\mathbb R}^3 \setminus \bar D_0,
\end{equation}
\begin{equation}
\label{v1b}
 \partial_r v_1 = k^2 u_1 d\;\mbox{on}\; 
\partial D_0
\end{equation}
with the Sommerfeld condition \eqref{Som} for $v_1$.

Unique solvalibity of the direct scattering problems in Sobolev and H\"older spaces 
is well known \cite{CK}, \cite{I}, \cite{T}.
For example, for any $d\in H^1(\partial D_0)$ there
is an unique radiating solution $v_0\in H^1(B_{\rho}\setminus \bar B_r$ (for any $\rho>R$) to the linearized direct scattering problem \eqref{vH}, \eqref{vb}.

The linearized inverse obstacle scattering problem is to find $D$ (or, equivalently, $d$) from $A(v_0)$ or $A(v_1)$.

\begin{corollary}

For a solution $d$ of the inverse soft obstacle problem we have
$$
\|d\|^2_{(0)}(\partial B_R)\leq
\frac{R^2}{k^2R^2+1}(\frac{2 e^2}{\pi} \varepsilon_1^2 
+ \frac{2}{\pi}e^{\frac{2}{R}} \varepsilon_2)+ 
 R^2\frac{\|d\|^2_{(1)}(\partial B_R)}{E+k}
 $$
 and
$$
\|d\|^2_{(0)}(\partial B_R)\leq
\frac{R^2}{k^2R^2+1}(\frac{2 e^2}{\pi}  \varepsilon_1^2 
+ \sqrt{\frac{2(k^2R^2+1)}{\pi k R}}e^{\frac{1}{R}} \varepsilon_2^{\frac{1}{2}})+ 
R^2\frac{\|d\|^2_{(1)}(\partial B_R)}{E+k}
 $$ 
 where $\varepsilon_1, \varepsilon_2, E$ are defined in Theorem 1.1 with $A$ is replaced by
 $A(v_0)$.
\end{corollary}

\begin{proof}

From \eqref{u0} by elementary calculations
$$
\partial_r u_0 (R\sigma) = \frac{R}{e^{ikR}}
\frac{ik e^{ikR}R -e^{ikR}}{R^2} = \frac{ikR-1}{R}
$$
Hence $|\partial_r u_0|^2=\frac{k^2R^2+1}{R^2}$ and
$$
\|v_0\|^2_{(1)}(\partial B_R) =
\frac{k^2R^2+1}{R^2}\|d\|^2_{(1)}(\partial B_R).
$$
So this Corollary follows from Theorem 1.1.

\end{proof}

\begin{corollary}

For a solution $d$ of the inverse hard obstacle problem we have
$$
\|d\|^2_{(0)}(\partial B_R)\leq
$$
$$\frac{k^2R^2+1}{k^2 R^2}
(\frac{e^2}{\pi}(3+\sqrt{5}) \varepsilon_1^2 
+   e^{\frac{2}{R}}\varepsilon_2+ 
R^2\frac{\|d\|_{(1)}^2(\partial B_R)}{E+k-2\sqrt{E+k}+1} ,
 $$
 where $\varepsilon_1, \varepsilon_2, E$ are defined in Theorem 1.2 with $A$ is replaced by
 $A(v_1)$.
\end{corollary}

\begin{proof}

Observe that, according to \eqref{u1}, \eqref{v1b},
$$
\|\partial_r v_1\|^2_{(1)}(\partial B_R)=
k^4\frac{R^2}{k^2R^2+1}\|d\|^2_{(1)}(\partial B_R).
$$
Now from \eqref{v1b}, \eqref{u1} and Theorem 1.2 we have
$$
\| d\|_{(0)}^2(\partial B_R)\leq
$$
$$
\frac{k^2R^2+1}{k^2R^2}(
\frac{e^2}{\pi}(3+\sqrt{5}) \varepsilon_1^2 
+  e^{\frac{2}{R}}\varepsilon_2)+ 
R^2\frac{\|d\|^2_{(1)}(\partial B_R)}{E+k-2\sqrt{E+k}+1} 
$$
and Corollary 4.2 follows.

\end{proof}

\section{Conclusion}

We think that increasing stability is an important feature of the  which leads to higher resolution 
of numerical algorithms. It is important to collect numerical evidence of this phenomenon. We tried to obtain most explicit forms of stability estimates to make them useful in particular for numerical solution of inverse scattering problems.

It is important to expand Lipschitz stability zone:
i.e. to replace the condition $n^2<t$  of Lemma 2.1 by the most natural condition $ n <\theta k$   with some $\theta<1$. Given numerous previous efforts, this seems to be a hard problem.  The results of this paper most likely imply similar increasing stability estimates when $B(0,R)$ is replaced by a strictly convex domain $D$. Indeed, one can represent the complement of such $D$ as the union
of the family of the exteriors of spheres whose radii and centers are contained in a bounded set
and use bounds \eqref{T1}, \eqref{T2} for these spheres. For general convex obstacles we do not expect such explicit and simple bounds.

Much more challenging is to show increasing stability for soft and hard (convex) obstacles.  
It is not clear even how to handle linearized 
problems near a sphere when the incident wave is traditional $e^{ik\xi\cdot x}$ with $|\xi|=1$. While the solution $u_0$ of the unperturbed soft scattering problem is well known \cite{CK}, \cite{T}, its is difficult to control zeros of its normal derivative on the boundary and hence to use \eqref{vb}.

\bibliographystyle{amsalpha}

 \end{document}